\documentclass[a4paper,12pt]{amsart}
\usepackage{amssymb}
\usepackage{ifthen}
\usepackage{graphicx}
\usepackage{amsmath}
\usepackage[T1]{fontenc} 

\nonstopmode \numberwithin{equation}{section}
\theoremstyle{plain}
\newtheorem{thm}{Theorem}[section]
\newtheorem{cor}{Corollary}[section]
\newtheorem{lem}{Lemma}[section]
\newtheorem{prop}{Proposition}

\newtheorem{conj}{Conjecture}
\newtheorem{innercustomthm}{Theorem}
\newenvironment{customthm}[1]
  {\renewcommand\theinnercustomthm{#1}\innercustomthm}
  {\endinnercustomthm}

\theoremstyle{definition}

\newtheorem{example}{Example}[section]
\newtheorem{prob}{Problem}
\newtheorem{rem}{Remark}[section]

\newcounter{minutes}
\divide\time by 60
\newcounter{hours}
\multiply\time by 60
\addtocounter{minutes}{-\time}

\newcounter {own}
\def\theown {\thesection       .\arabic{own}}

\newenvironment{pf}[1][]{%
 \vskip 3mm
 \noindent
 \ifthenelse{\equal{#1}{}}%
  {{\slshape Proof. }}%
  {{\slshape #1.} }%
 }%
{\qed\bigskip}

\newcounter{alphabet}
\newcounter{tmp}

\def\be{\begin{equation}}
\def\ee{\end{equation}}

\newcommand{\bee}{\begin{enumerate}}
\newcommand{\eee}{\end{enumerate}}

\newcommand{\blem}{\begin{lem}}
\newcommand{\elem}{\end{lem}}
\newcommand{\bthm}{\begin{thm}}
\newcommand{\ethm}{\end{thm}}
\newcommand{\bcor}{\begin{cor}}
\newcommand{\ecor}{\end{cor}}
\newcommand{\beg}{\begin{examp}}
\newcommand{\eeg}{\end{examp}}
\newcommand{\begs}{\begin{examples}}
\newcommand{\eegs}{\end{examples}}
\newcommand{\bdefe}{\begin{defin}}
\newcommand{\edefe}{\end{defin}}
\newcommand{\bprob}{\begin{prob}}
\newcommand{\eprob}{\end{prob}}
\newcommand{\bei}{\begin{itemize}}
\newcommand{\eei}{\end{itemize}}

\newcommand{\bcon}{\begin{conj}}
\newcommand{\econ}{\end{conj}}
\newcommand{\bcons}{\begin{conjs}}
\newcommand{\econs}{\end{conjs}}
\newcommand{\bprop}{\begin{prop}}
\newcommand{\eprop}{\end{prop}}
\newcommand{\br}{\begin{rem}}
\newcommand{\er}{\end{rem}}
\newcommand{\brs}{\begin{rems}}
\newcommand{\ers}{\end{rems}}
\newcommand{\bo}{\begin{obser}}
\newcommand{\eo}{\end{obser}}
\newcommand{\bos}{\begin{obsers}}
\newcommand{\eos}{\end{obsers}}
\newcommand{\bpf}{\begin{pf}}
\newcommand{\epf}{\end{pf}}
\newcommand{\ba}{\begin{array}}
\newcommand{\ea}{\end{array}}
\newcommand{\beq}{\begin{eqnarray}}
\newcommand{\beqq}{\begin{eqnarray*}}
\newcommand{\eeq}{\end{eqnarray}}
\newcommand{\eeqq}{\end{eqnarray*}}

\begin{document}

\title{Coefficient estimates for certain subclass of analytic functions defined by subordination}

\author{Nirupam Ghosh}
\address{Nirupam Ghosh,
Department of Mathematics,
Indian Institute of Technology Kharagpur,
Kharagpur-721 302, West Bengal, India.}
\email{nirupamghoshmath@gmail.com}

\author{A. Vasudevarao}
\address{A. Vasudevarao,
Department of Mathematics,
Indian Institute of Technology Kharagpur,
Kharagpur-721 302, West Bengal, India.}
\email{alluvasu@maths.iitkgp.ernet.in}


\def\thefootnote{}
\footnotetext{ {\tiny File:~\jobname.tex,
printed: \number\year-\number\month-\number\day,
          \thehours.\ifnum\theminutes<10{0}\fi\theminutes }
} \makeatletter\def\thefootnote{\@arabic\c@footnote}\makeatother

\begin{abstract}
In this article we determine the coefficient bounds for functions in certain subclasses of analytic functions defined by subordination which are related to the well-known classes of starlike and convex functions. The main results deal with some open problems proposed by
Q.H. Xu {\it et al}.(\cite{Xu-Cai-Srivastava-2013}, \cite{Xu-Lv-Luo-Srivastava-2013}). An application of Jack lemma for certain subclass of starlike functions has been discussed.
\end{abstract}

\maketitle
\pagestyle{myheadings}
\markboth{ Nirupam Ghosh and A. Vasudevarao }{Coefficient estimates for certain subclass of analytic functions}

\section{introduction}
Let $\mathcal{A}$ denote the family of analytic functions $f$ in the unit disk $\mathbb{D}:=\{z\in\mathbb{C}:\,|z|<1\}$ normalized by
$f(0)=0=f'(0)-1 $. If $f\in \mathcal{A}$ then $f$ has the following representation
\begin{equation}\label{niru vasu p1 e001}
f(z)=z+\sum_{n=2}^{\infty}a_n z^n.
\end{equation}
A function $f$ is said to be univalent in  a domain $\Omega\subseteq\mathbb{C}$ if it is injective in $\Omega$. Let $\mathcal{S}$ denote the
class of univalent functions in $\mathcal{A}$.
A function $f\in\mathcal{A}$ is in the class  $\mathcal{S}^*(\alpha)$, called starlike functions of order $\alpha$, if
$$
{\rm Re\,}\left(\frac{zf'(z)}{f(z)}\right)> \alpha\quad\mbox{ for }  \quad z\in\mathbb{D}
$$
and in the class $\mathcal{C}(\alpha)$, called convex functions of order $\alpha$, if
$$
{\rm Re\,}\left(1+\frac{z f''(z)}{f'(z)}\right)> \alpha \quad\mbox{ for }  \quad z\in\mathbb{D}.
$$
Clearly the classes $\mathcal{S}^*:=\mathcal{S}^*(0)$ and $\mathcal{C}:=\mathcal{C}(0)$ are the well- known classes of starlike and convex functions respectively.
It is well -known that $\mathcal{C}\subsetneq\mathcal{S}^*\subsetneq\mathcal{S}$. A function $f\in\mathcal{A}$ is in the class $\mathcal{SP}(\alpha)$, called $\alpha$-Spiral functions, if
$$
{\rm Re\,}\left(e^{i\alpha} \frac{zf'(z)}{f(z)}\right)> 0 \quad \mbox{for}\quad z\in\mathbb{D}.
$$
The class $\mathcal{SP}(\alpha)$ has been introduced by \v{S}pa\v{c}ek \cite{Spacek-33} in 1933.

Let $f$ and  $g$ be analytic functions in the unit disk $\mathbb{D}$. A function $f$ is said to be subordinate to $g$, written  as
$f\prec g$ or $f(z)\prec g(z)$, if there exists an analytic function $\omega: \mathbb{D} \rightarrow \mathbb{D}$ with $\omega(0)=0$ such that $f(z)=g(\omega(z))$.
If $g$ is univalent, then $f\prec g$ if and only if $f(0)=g(0)$ and $f(\mathbb{D})\subseteq g(\mathbb{D})$.

For $A, B \in \mathbb{C}$ with $|B|\leq 1$, let $\mathcal{S}^*[A, B]$ denote the class of functions $f \in \mathcal{A}$ which satisfy the following subordination relation
$$
\frac{zf'(z)}{f(z)}\prec \frac{1 + Az}{1 + Bz} \quad \mbox{for} \quad z \in \mathbb{D}.
$$
Without loss of generality we may assume that $B$ is a real. In view of $S^*[A, B] = S^*[-A, -B]$, we can consider $-1 \leq B \leq 0.$ For particular choice of parameters $A$ and $B$, we can obtain $ \mathcal{S}^* := \mathcal{S}^*[1 , -1]$ and $\mathcal{S}^*(\alpha) := \mathcal{S}^*[1 - 2\alpha, -1]$. If we choose $A = e^{-2i\alpha}$ and $B = -1$ then  $\mathcal{SP}(\alpha) :=S^*[e^{-2i\alpha}, -1]$.

Nasr and Aouf \cite{Nasr-Aouf-1983,Nasr-Aouf-1983a,Nasr-Aouf-1985} and Wiatrowski \cite{Wiatrowski-1970} extended the classes $\mathcal{S}^*(\alpha)$
and $\mathcal{C}(\alpha)$ by introducing  $\mathcal{S}^*(\gamma)$ and $\mathcal{C}(\gamma)$,
the class of starlike functions of complex order $\gamma$ and   the class of convex functions of
complex order  $\gamma$  respectively.
More preciously, a function $f\in\mathcal{A}$ is said to be in the class $\mathcal{S}^*(\gamma)$,
if it satisfies the following condition
$$
{\rm Re\,}\left(1+\frac{1}{\gamma}\left(\frac{zf'(z)}{f(z)}-1\right)\right)>0\quad \mbox{for} \quad z\in\mathbb{D}\quad\mbox{ and } \quad  \gamma\in\mathbb{C}\setminus\{0\}.
$$
Similarly, a function $f\in\mathcal{A}$ is said to be in the class $\mathcal{C}(\gamma)$, if it satisfies the following condition
$$
{\rm Re\,}\left(1+\frac{1}{\gamma}\left(\frac{z f''(z)}{f'(z)}\right)\right)>0 \quad \mbox{for} \quad z\in\mathbb{D} \quad \mbox{ and } \quad  \gamma\in\mathbb{C}\setminus\{0\}.
$$
The function classes $\mathcal{S}^*(\gamma)$ and  $\mathcal{C}(\gamma)$ have been extensively studied by many
authors (for example, see \cite{Altintas-Irmak-Srivastava-1995,Altintas-Ozkan-2001,Altintas-Srivastava-2001,Altintas-Ozkan-Srivastava-2000,
Altintas-Ozkan-Srivastava-2001}).
For  fixed $\beta>1$, the classes $\mathcal{M}(\beta)$ and $\mathcal{N}(\beta)$ are defined by
\begin{eqnarray*}
\mathcal{M}(\beta)&:=&\left\{f\in\mathcal{A}:\, {\rm Re\,}\left(\frac{zf'(z)}{f(z)}\right)<\beta \quad\mbox{ for } z\in\mathbb{D}\right\}\quad\mbox{ and }\\
\mathcal{N}(\beta)&:=&\left\{f\in\mathcal{A}:\, {\rm Re\,}\left(1+\frac{zf''(z)}{f'(z)}\right)<\beta \quad\mbox{ for } z\in\mathbb{D}\right\}.
\end{eqnarray*}
If we choose $\gamma = (1- \beta)$ then the class $ \mathcal{M}(\beta):=\mathcal{S}^*(1 - \beta)$ and
$\mathcal{N}(\beta):=\mathcal{C}(1 - \beta)$.
The classes  $\mathcal{M}(\beta)$ and $\mathcal{N}(\beta)$ have been extensively discussed by
Obradovic {\it et al.} \cite{Obradovic-Ponnusamy-Wirths-2013} and Firoz Ali and Vasudevarao \cite{Firoz-Vasudeva-2015}.

In 2007, Altintas {\it et al.} \cite{Altintas-Irmak-Owa-Srivastava-2007} introduced the classes $\mathcal{S}c(\gamma,\lambda,\beta)$ and
$\mathcal{B}(\gamma,\lambda,\beta,\mu)$. A function $f\in\mathcal{A}$ is in the class $\mathcal{S}c(\gamma,\lambda,\beta)$ for
$\gamma\in\mathbb{C}\setminus\{0\},0\leq\lambda\leq1$ and $ 0\leq\beta<1$ if it satisfies the following condition
$$
{\rm Re\,}\left(1+\frac{1}{\gamma} \left(\frac{z (\lambda zf'(z) + (1-\lambda) f(z))'}{(\lambda zf'(z) +
(1-\lambda) f(z))}-1  \right)\right)>\beta\quad\mbox{ for } \quad z\in\mathbb{D}.
$$
Clearly $\mathcal{S}^*(\gamma)$ := $\mathcal{S}c(\gamma,0,0)$   and   $\mathcal{C}(\gamma)$ := $\mathcal{S}c(\gamma,1,0)$.
A function $w=f(z)$ belongs to $\mathcal{A}$  is said to be in the class $\mathcal{B}(\gamma,\lambda,\beta,\mu)$ if it satisfies the following
non-homogeneous Cauchy-Euler differential equation
$$
z^2\frac{d^2w}{dz^2} + \mu z \frac{dw}{dz} + \mu (\mu +1)w = (\mu + 1)(\mu + 2)g(z),
$$
where  $g\in\mathcal{S}c(\gamma,\lambda,\beta)$ and  $\mu\in\mathbb{R}\setminus(-\infty , -1]$.
In \cite{Altintas-Irmak-Owa-Srivastava-2007}, the authors obtained the coefficient bounds for functions in the classes $\mathcal{S}c(\gamma,\lambda,\beta)$
and $\mathcal{B}(\gamma,\lambda,\beta,\mu)$ but the results were not sharp.

In 2011,  Srivastava {\it et al.} \cite{Srivastava-Altintas-Serenbarg-2011} introduced the classes $\mathcal{S}(\lambda,\gamma,A,B)$ and
$\mathcal{K}(\lambda,\gamma,A,B,m,\mu)$. A function $f\in\mathcal{A}$ is in the class  $\mathcal{S}(\lambda,\gamma,A,B)$ if it satisfies the following subordination condition

$$
{\rm }1+\frac{1}{\gamma}\left(\frac{z (\lambda zf'(z) + (1-\lambda) f(z))'}{(\lambda zf'(z) +
(1-\lambda) f(z))}-1\right)\prec \frac{1 + Az}{1 + Bz}\quad\mbox{ for } \quad z\in\mathbb{D},
$$
where $\gamma\in\mathbb{C}\setminus\{0\}, 0\leq\lambda\leq1$ and $-1\leq B < A \leq1$. Similarly, a function $w =f(z)$ belongs to $\mathcal{A}$
is said to be in the class $\mathcal{K}(\lambda,\gamma,A,B,m,\mu$) if it satisfies the following non-homogeneous Cauchy-Euler type
differential equation of order $m$
\begin{equation}\label{niru vasu p1 e0014}
z^m\frac{d^mw}{dz^m} + \dbinom{m}{1} (\mu + m -1)z^{m-1}\frac{d^{m-1}w}{dz^{m-1}} + \cdots+ \dbinom{m}{m}\prod_{j=0}^{m-1}(\mu + j)w =
g(z)\prod_{j=0}^{m-1}(\mu + j + 1),
\end{equation}
where  $g\in \mathcal{S}(\lambda,\gamma,A,B)$, $\mu\in\mathbb{R}\setminus(-\infty , -1]$ and $m\in\mathbb{N}\setminus\{1\}$.
For particular choice of the parameters $A$ and $B$, we obtain  $\mathcal{S}c(\gamma,\lambda,\beta) := \mathcal{S}(\lambda,\gamma,1-2\beta,-1)$,
$\mathcal{S}^*(\gamma):= \mathcal{S}(0,\gamma,1,-1)$ and  $\mathcal{C}(\gamma) := \mathcal{S}(1,\gamma,1,-1)$. The coefficient bounds for functions
in the classes $\mathcal{S}(\lambda,\gamma,A,B)$ and $\mathcal{K}(\lambda,\gamma,A,B,m,\mu)$ have been investigated
by Srivastava {\it et al.} \cite{Srivastava-Altintas-Serenbarg-2011} but the results are not sharp.
Recently, Q-H Xu {\it et al.} \cite{Xu-Cai-Srivastava-2013} obtained the following sharp coefficient bounds for functions
in classes $\mathcal{S}(\lambda,\gamma,A,B)$ and $\mathcal{K}(\lambda,\gamma,A,B,m,\mu)$ with some restriction on the parameters.

\begin{customthm}{A}\cite{Xu-Cai-Srivastava-2013}
Let $f\in\mathcal{S}(\lambda,\gamma,A,B)$ be given by (\ref{niru vasu p1 e001}), where $\gamma\in\mathbb{C}\setminus\{0\}, 0\leq\lambda\leq1$ and
$-1\leq B < A \leq1$. If
$$
|\gamma(A-B)-B(n-2)|\geq (n-2),
$$
then
\begin{equation}\label{niru vasu p1 e0010}
|a_n|\leq \frac{ \prod_{j=0}^{n-2}{|(A-B)\gamma-jB|}}{{(1+\lambda(n-1))}{(n-1)!}}, \quad n\in\mathbb{N}\setminus\{1\}
\end{equation}
and the estimates in (\ref{niru vasu p1 e0010}) are sharp.
\end{customthm}

\begin{customthm}{B}\cite{Xu-Cai-Srivastava-2013}
Let $f\in\mathcal{K}(\lambda,\gamma,A,B,m,\mu)$ be given by (\ref{niru vasu p1 e001}), where
$\gamma\in\mathbb{C}\setminus\{0\}, 0\leq\lambda\leq1, -1\leq B < A \leq1$, $\mu\in\mathbb{R}\setminus(-\infty , -1]$ and $m\in\mathbb{N}\setminus\{1\}$. If
$$
|\gamma(A-B)-B(n-2)|\geq (n-2),
$$
then
\begin{equation}\label{niru vasu p1 e0015}
|a_n|\leq \frac{{ \prod_{j=0}^{n-2}{|(A-B)\gamma-jB|}{\prod_{j=0}^{m-1}(\mu +j+1)}}}{{(1+\lambda(n-1))}{(n-1)!}{\prod_{j=0}^{m-1}(\mu+j+n)}},
\quad n,m\in\mathbb{N}\setminus\{1\}
\end{equation}
and the estimates in (\ref{niru vasu p1 e0015}) are sharp.
\end{customthm}

In 2013, Xu {\it et al.} \cite{Xu-Cai-Srivastava-2013} proposed the following two problems concerning the coefficient bounds for functions in the  class $\mathcal{S}(\lambda,\gamma,A,B)$.

\begin{prob}\label{niru vasu p1 problem 03}
If the function $f\in\mathcal{S}(\lambda,\gamma,A,B)$ is given by (\ref{niru vasu p1 e001}) with $\gamma\in\mathbb{C}\setminus\{0\},0\leq\lambda\leq1$
and $-1\leq B < A \leq1$ then prove or disprove that
\begin{equation}\label{niru vasu p1 e0011}
|a_n|\leq \frac{ \prod_{j=0}^{n-2}{|(A-B)\gamma-jB|}}{{(1+\lambda(n-1))}{(n-1)!}}, \quad n\in\mathbb{N}\setminus\{1\}.
\end{equation}
\end{prob}

\begin{prob}\label{niru vasu p1 problem 04}
If the coefficient estimates in (\ref{niru vasu p1 e0011}) do hold true then prove or disprove that these estimates are sharp.

\end{prob}

In 2013, Xu {\it et al.} \cite{Xu-Lv-Luo-Srivastava-2013} considered the class $\mathcal{S}^{\beta}(A,B)$ by the condition that a function
$f\in\mathcal{A}$ is in the class $\mathcal{S}^{\beta}(A,B)$ if it satisfies

$$
(1+i\tan\beta)\frac{zf'(z)}{ f(z)}-i\tan\beta\prec \frac{1 + Az}{1 + Bz}\quad\mbox{ for } \quad z\in\mathbb{D},
$$
where $-\pi/2<\beta<\pi/2$ and $-1\leq B < A \leq1$ and
obtained the following coefficient bounds for functions in this class.

\begin{customthm}{C}\cite{Xu-Lv-Luo-Srivastava-2013}
Let $f\in\mathcal{S}^{\beta}(A,B)$ be given by (\ref{niru vasu p1 e001}) with $-\pi/2<\beta<\pi/2$, $-1\leq B < A \leq1$ and
$n\in\mathbb{N}\setminus\{1\}$. Suppose also that
\begin{equation}\label{niru-vasu-p1-e005ab}
(A- (n-1)B)^2 \cos^2{\beta} + (n-2)^2 (B^2\sin^2{\beta}-1)\geq 0.
\end{equation}
Then
\begin{equation}\label{niru vasu p1  e005}
|a_n|\leq \prod_{j=0}^{n-2}\bigg(\frac{|(A-B)e^{-i\beta}\cos{\beta}-jB|}{j+1}\bigg), \quad n\in\mathbb{N}\setminus\{1\}
\end{equation}
and the  estimates in (\ref{niru vasu p1 e005}) are sharp .
\end{customthm}
We note that  Theorem C  is proved under the additional assumption  (\ref{niru-vasu-p1-e005ab}). In the same paper the authors proposed
the following two problems concerning the coefficient bounds for functions in class $\mathcal{S}^{\beta}(A,B)$
without assuming the additional condition  (\ref{niru-vasu-p1-e005ab}).

\begin{prob}\label{niru vasu p1 problem 01}
If the function $f\in\mathcal{S}^{\beta}(A,B)$ is given by (\ref{niru vasu p1 e001}) with $-\pi/2<\beta<\pi/2$ and  $-1\leq B < A \leq1$,
then prove or disprove that
\begin{equation}\label{niru vasu p1 e006}
|a_n|\leq \prod_{j=0}^{n-2}\bigg(\frac{|(A-B)e^{-i\beta}\cos{\beta}-jB|}{j+1}\bigg), \quad n\in\mathbb{N}\setminus\{1\}.
\end{equation}
\end{prob}

\begin{prob}\label{niru vasu p1 problem 02}
If the coefficient estimates in (\ref{niru vasu p1 e006}) do hold true  then prove or disprove that these estimates are sharp.

\end{prob}

It is interesting to note that if we choose $\lambda=0$ and $\gamma = {1}/{(1+i\tan\beta)} $ then the class $\mathcal{S}(\lambda,\gamma,A,B)$ reduces
 to $\mathcal{S}^{\beta}(A,B)$. Hence it is sufficient to study Problem \ref{niru vasu p1 problem 03} and Problem \ref{niru vasu p1 problem 04} for
 functions in the class $\mathcal{S}(\lambda,\gamma,A,B)$. \\

The  problem of coefficient estimates  is one of the most exciting problem in the theory of univalent  functions.
For $f\in\mathcal{S}$ of the form (\ref{niru vasu p1 e001}), it was proved that $|a_2|\leq 2$ and proposed a conjecture $|a_n|\leq n$ for $n \geq 3$ by Bieberbach in $1916$. This celebrated conjecture was proved affirmatively by
Branges in  $1984$. This  motivates us to   determine the coefficient bounds for functions in some subclasses of analytic functions which are defined by the subordination and these classes are related to the well-known classes of starlike and convex functions.

The  main aim of this paper is to  attempt the  aforementioned problems in much detailed. In fact,
the main results of this paper deal with some open problems proposed by Q.H. Xu {\it et al}.(\cite{Xu-Cai-Srivastava-2013}, \cite{Xu-Lv-Luo-Srivastava-2013}).

Before  proving  our main results, we recall   the following lemma due to  Xu {\it et al.} \cite{Xu-Cai-Srivastava-2013}.

\begin{lem}\label{niru vasu p1 l001}\cite{Xu-Cai-Srivastava-2013}
Let the parameters A,B $\lambda ,\gamma$ and $m$ satisfy $\gamma\in\mathbb{C}\setminus\{0\},0\leq\lambda\leq1,-1\leq B < A \leq1$ and
$m\in\mathbb{N}\setminus\{1\}$. If
$|\gamma(A-B)-B(m-2)|\geq (m-2)$, then
\begin{align*}
|\gamma|^2(A - B)^2 &+ \sum_{k=2}^{m-1}\left[{\frac{\left||\gamma(A - B) - B(k-1)|^2 - (k-1)^2\right|}{((k-1)!)^2}}\right]
\prod_{j=0}^{k-2}|\gamma(A - B) - jB|^2\\[3mm]
&=\frac {\prod_{j=0}^{m-2}|\gamma(A -B)-Bj|^2}{((m-2)!)^2}.
\end{align*}
\end{lem}

\section{Coefficient estimates}

In this section, we will estimate the modulus of the coefficients of function of the form (\ref{niru vasu p1 e001}), which belong to the class of $\mathcal{S}(\lambda,\gamma,A,B)$ and $\mathcal{K}(\lambda,\gamma,A,B,m,\mu)$. Moreover, the inequalities obtained will be examined in terms of sharpness.

\begin{thm}\label{niru vasu p1 t001}
Let $f\in\mathcal{S}(\lambda,\gamma,A,B)$ be of the form (\ref{niru vasu p1 e001}),
where $\gamma\in\mathbb{C}\setminus\{0\},0\leq\lambda\leq1, -1\leq B < A \leq1$ and
$n\in\mathbb{N}\setminus\{1\}$ be fixed and define $A_k=|\gamma(A-B)-B(k-1)|-(k-1)$.

\begin{enumerate}
\item[(i)] If $A_2\leq{0}$, then
\begin{equation}\label{niru vasu p1 e0020}
|a_n|\leq\frac{|\gamma|(A-B)}{(n-1)(1+\lambda(n-1))}.
\end{equation}

\item[(ii)] If $A_{n-1}\geq {0}$, then
\begin{equation}\label{niru vasu p1 e0022}
|a_n|\leq\frac{\prod_{j=0}^{n-2}|\gamma(A-B)-jB|}{(n-1)!(1+\lambda(n-1))}.
\end{equation}

\item[(iii)] If $A_k\geq {0}$ and $A_{k+1}\leq{0}$ for  $k= 2,3,\ldots,n-2$,  then
\begin{equation}\label{niru vasu p1 e0021}
|a_n|\leq\frac{\prod_{j=0}^{k-1}|\gamma(A-B)-jB|}{(k-1)!(n-1)(1+\lambda(n-1))}.
\end{equation}

\end{enumerate}
The estimates in (\ref{niru vasu p1 e0020}) and (\ref{niru vasu p1 e0022}) are sharp.
\end{thm}

\begin{pf}
The proof of part (ii) can be found in \cite{Xu-Cai-Srivastava-2013}. But for the sake of completeness of the result, we include it here.
Let $f\in\mathcal{S}(\lambda,\gamma,A,B)$. Then there exists an analytic function $\omega(z)$ in $\mathbb{D}$ with $\omega(0)=0$ and $|\omega(z)|<1 $ such that
\begin{equation}\label{niru vasu p1 e0024}
1+\frac{1}{\gamma}\left(\frac{z (\lambda zf'(z) + (1-\lambda) f(z))'}{(\lambda zf'(z) + (1-\lambda) f(z))}-1\right)= \frac{1 + A\omega(z)}{1 + B\omega(z)}.
\end{equation}
Using the series expansion (\ref{niru vasu p1 e001}) of $f(z)$ in (\ref{niru vasu p1 e0024}) and then after simplification we obtain
\begin{align*}
\sum_{k=2}^{\infty}(k-1)&(1+\lambda (k -1))a_kz^k\\
& = \left(\gamma(A-B)z + \sum_{k=2}^{\infty}(\gamma(A-B)- B(k-1))(1+\lambda (k -1))a_k z^k \right)\omega(z)
\end{align*}
which can be written as
\begin{align*}
\sum_{k=2}^{n}(k-1)&(1+\lambda (k -1))a_kz^k + \sum_{k=n+1}^{\infty}b_kz^k\\
& = \bigg(\gamma(A-B)z + \sum_{k=2}^{n-1}(\gamma(A-B)- B(k-1))(1+\lambda (k -1))a_k z^k \bigg)\omega(z)
\end{align*}
for certain coefficients $b_k$.
Since $|\omega(z)|<1$,   an application of  Parseval's theorem gives
\begin{align*}
\sum_{k= 2}^{n}(k-1)^2&(1+\lambda (k -1))^2|a_k|^2 + \sum_{k=n+1}^{\infty}|b_k|^2\\
& \leq |\gamma|^2(A-B)^2 + \sum_{k=2}^{n-1}\big(|\gamma(A-B)- B(k-1)|^2\big)(1+\lambda (k -1))^2|a_k|^2
\end{align*}
and therefore


\begin{align}\label{niru vasu p1 e0023}
(n-1)^2(1&+\lambda (n -1)^2 )|a_n|^2 \leq|\gamma|^2(A-B)^2\\
& + \sum_{k=2}^{n-1}\bigg(|\gamma(A-B)- B(k-1)|^2 - (k-1)^2\bigg)(1+\lambda (k -1))^2|a_k|^2.\nonumber
\end{align}

For $n=2 $, it follows from (\ref{niru vasu p1 e0023}) that
\begin{equation}\label{niru vasu p1 e0026}
|a_2| \leq \frac {|\gamma|(A-B)}{1+ \lambda}.
\end{equation}
Note that if $A_k\geq 0$ then $A_{k-1}\geq 0$ for $k=2,3,\ldots$, because
$$
|\gamma(A-B)- (k-2)B| \geq |\gamma(A-B)- (k-1)B| -|B| \geq (k-1)-1 =k-2.
$$
Again, if $A_k\leq 0$ then $A_{k+1}\leq 0$ for $k=2,3,\ldots$, because
$$
|\gamma(A-B)- kB| \leq |\gamma(A-B)- (k-1)B| +|B| \leq (k-1)+1 =k.
$$

If $A_2\leq 0$  then from the above discussion we can conclude that $ A_k\leq 0$ for all $k> {2} $. It follows from  (\ref{niru vasu p1 e0023}) that
\begin{equation*}
(n-1)^2(1+\lambda (n -1)^2)|a_n|^2\leq {|\gamma|^2(A-B)^2}
\end{equation*}
and consequently
\begin{equation}\label{niru-vasu- p1- e0026aa}
|a_n|\leq \frac{|\gamma|(A-B)}{(n-1)(1+\lambda(n-1))}.
\end{equation}
Equality in (\ref{niru-vasu- p1- e0026aa})  is attained  for the  functions $f_n(z)$  where $f_n(z)$  satisfies  the
following  differential equation
$$
\lambda zf_n'(z)+ (1-\lambda)f_n(z) = z(1 + Bz^{n-1})^\frac{\gamma(A-B)}{B(n-1)}.
$$

Next, let $A_{n-1}\geq{0}$. Then from the above discussion we have $A_{2},A_{3}, A_{4},\ldots ,A_{n-2}\geq{0}$.
 From (\ref{niru vasu p1 e0026}) it is clear that  (\ref{niru vasu p1 e0022}) is true for $n=2$. Suppose that  (\ref{niru vasu p1 e0022})
 is true for $k = 2, 3,\ldots,n-1$. Then using the induction hypothesis, it follows from (\ref{niru vasu p1 e0023}) that
\begin{align*}
&(n-1)^2(1+\lambda (n -1))^2 |a_n|^2\\
&\leq |\gamma|^2(A-B)^2 + \sum_{k=2}^{n-1}\left(||\gamma(A-B)- B(k-1)|^2 - (k-1)^2| \right)(1+\lambda (k -1))^2|a_k|^2\\
&\leq|\gamma|^2(A-B)^2 + \sum_{k=2}^{n-1}\left(||\gamma(A-B)- B(k-1)|^2 - (k-1)^2| \right)(1+\lambda (k -1))^2\times\\  &\qquad\frac{\prod_{j=0}^{k-2}|\gamma(A-B)-jB|^2}{((k-1)!)^2(1+\lambda(k-1))^2}.
\end{align*}
An application of Lemma \ref{niru vasu p1 l001} shows that
\begin{equation*}
(n-1)^2(1+\lambda (n -1))^2 |a_n|^2\leq\frac{\prod_{j=0}^{n-2}|\gamma(A-B)-jB|^2}{((n-2)!)^2}
\end{equation*}
and consequently,
$$
|a_n|\leq\frac{\prod_{j=0}^{n-2}|\gamma(A-B)-jB|}{((n-1)!)(1+\lambda(n-1))}.
$$
By the mathematical induction, (\ref{niru vasu p1 e0022}) is true for all $n\geq 2$.
The equality in (\ref{niru vasu p1 e0022}) is attained for the following function
$$
f(z)=
\begin{cases}
\frac{\lambda - 1}{\lambda}\int_{0}^{z}\frac{t^{\frac{\lambda - 1}{\lambda}}}{\lambda (1 + Bt)^{\frac{B - A}{B}\gamma}} \,dt & \mbox{ for }\quad B \neq 0, \lambda \neq 0\\[4mm]
\frac{z}{ (1 + Bz)^{\frac{B-A}{B}\gamma}} &\mbox{ for }\quad B \neq 0, \lambda = 0\\[5mm]
\frac{1}{\lambda}\int_{0}^{z}t^{\frac{1 - \lambda}{\lambda}} e^ {A \gamma t} \,dt &\mbox{ for }\quad B = 0, \lambda \neq 0\\[3mm]
z e^{A \gamma z} &\mbox{ for }\quad B = 0, \lambda = 0.
\end{cases}
$$


Now if we assume that $A_k\geq{0}$ and $A_{k+1}\leq{0}$ for $k= 2,3,\ldots,n-2$. Then $A_{2},A_{3}, A_{4},\ldots ,A_{k-1}\geq{0}$ and
$A_{k+2},A_{k+3},\ldots ,A_{n-2}\leq{0}$. Using (\ref{niru vasu p1 e0022}) and Lemma \ref{niru vasu p1 l001}
in (\ref{niru vasu p1 e0023}), we obtain
\begin{align*}
&(n-1)^2(1+\lambda (n -1)^2 )|a_n|^2\\
&\leq |\gamma|^2(A-B)^2 + \sum_{l=2}^{k}\left(||\gamma(A-B)- B(l-1)|^2 - (l-1)^2| \right)(1+\lambda (l -1))^2|a_l|^2\\
&\leq|\gamma|^2(A-B)^2 + \sum_{l=2}^{k}\big(||\gamma(A-B)- B(l-1)|^2 - (l-1)^2| \big)\frac{\prod_{j=0}^{l-2}|\gamma(A-B)-jB|^2}{((l-1)!)^2}\\
&= \frac{\prod_{j=0}^{k -1}|\gamma(A-B)-jB|^2}{((k-1)!)^2},
\end{align*}
from which (\ref{niru vasu p1 e0021}) follows.


\end{pf}

\begin{thm}

Let $f\in\mathcal{K}(\lambda,\gamma,A,B,m,\mu)$ be of the form (\ref{niru vasu p1 e001}) and
$\gamma\in\mathbb{C}\setminus\{0\},0\leq\lambda\leq1,-1\leq B < A \leq1, m\in\mathbb{N}\setminus\{1\}\quad and \quad
\mu\in\mathbb{R}\setminus(-\infty,-1]$. Define $A_k=|\gamma(A-B)-B(k-1)|-(k-1)$.
\begin{enumerate}
\item[(i)]
If $A_2\leq{0}$, then
\begin{equation}\label{niru vasu p1 e0025}
|a_n|\leq\frac{|\gamma|(A-B)}{(n-1)(1+(n-1)\lambda)}\frac{\prod_{j=0}^{m-1}(\mu + j+ 1)}{\prod_{j=0}^{m-1}(\mu+j+n)}.
\end{equation}

\item[(ii)]

If $A_{n-1} \geq {0}$, then

\begin{equation}\label{niru vasu p1 e0030}
|a_n|\leq\frac{\prod_{j=0}^{n-2}|\gamma(A-B) - jB|}{(n-1)!(1 + \lambda(n-1))}\frac{\prod_{j=0}^{m-1}(\mu + j +1)}{\prod_{j=0}^{m-1}(\mu + j+n)}.
\end{equation}

\item[(iii)]

If $A_k\geq{0} \quad and \quad A_{k+1}\leq{0} \quad for\quad  k= 2, 3,\ldots,n-2$, then

\begin{equation}\label{niru vasu p1 e0035}
|a_n|\leq\frac{\prod_{j=0}^{k-1}|\gamma(A-B) - jB|}{(n-1)(k-1)!(1 + \lambda(n-1))}\frac{\prod_{j=0}^{m-1}(\mu + j +1)}{\prod_{j=0}^{m-1}(\mu + j+n)}.
\end{equation}

\end{enumerate}
The estimates in (\ref{niru vasu p1 e0025}) and (\ref{niru vasu p1 e0030}) are sharp.
\end{thm}

\begin{pf}
Let $f\in\mathcal{K}(\lambda,\gamma,A,B,m,\mu)$ be of the form (\ref{niru vasu p1 e001}). Then there exists $g\in\mathcal{S}(\lambda,\gamma,A,B)$
of the form  $g(z) = z + \sum_{n=2}^{\infty}b_nz^n$ such that (\ref{niru vasu p1 e0014}) holds. By comparing the coefficients on both sides of
(\ref{niru vasu p1 e0014}), we obtain

 $$
 a_n = \left(\frac{\prod_{j=0}^{m-1}(\mu + j + 1)}{\prod_{j=0}^{m-1}(\mu + j + n)}\right)b_n,
 $$
 where $m,n\in\mathbb{N}\setminus\{1\}$ and $\mu\in\mathbb{R}\setminus(-\infty , -1]$. Then the desired results follow from Theorem \ref{niru vasu p1 t001}.
 The sharpness of (\ref{niru vasu p1 e0025}) and (\ref{niru vasu p1 e0030}) easily follow from the sharpness of  (\ref{niru vasu p1 e0020}) and
  (\ref{niru vasu p1 e0022}).
%
\end{pf}

\begin{cor}
Let $f\in\mathcal{S}c(\gamma,\lambda,\beta)$ be given by (\ref{niru vasu p1 e001}).
\begin{enumerate}
\item[(i)]
If
$|2\gamma(1- \beta) + 1 |\leq 1$, then
\begin{equation}\label{niru vasu p1 e040}
|a_n|\leq\frac{2|\gamma|(1 - \beta)}{(n-1)(1 + (n-1)\lambda)}.
\end{equation}
The equality in (\ref{niru vasu p1 e040}) occurs for the solution of equation
\begin{equation*}
\lambda zf_n'(z)+ (1-\lambda)f_n(z) = z(1 - z^{n-1})^\frac{-2\gamma(1 - \beta)}{(n-1)}.
\end{equation*}

\item[(ii)] If
$|2\gamma(1-\beta)+(n-2)|\geq (n-2)$, then
\begin{equation}\label{niru vasu p1 e045}
|a_n|\leq\frac{\prod_{j=0}^{n-2}|2\gamma(1 - \beta)+j|}{(n-1)!(1 + (n-1)\lambda)}.
\end{equation}
The inequality (\ref{niru vasu p1 e045})  is  sharp.

\item[(iii)] If
$|2\gamma(1-\beta)+(k-1)|\geq (k-1)$, then
$$
|a_n|\leq\frac{\prod_{j=0}^{k-1}|2\gamma(1 - \beta)+j|}{(n-1)(k-1)!(1 + (n-1)\lambda)}.
$$
\end{enumerate}

%
%
\end{cor}

\begin{cor}
Let $f\in\mathcal{B}(\gamma,\lambda,\beta,\mu)$ be given by (\ref{niru vasu p1 e001}).
\begin{enumerate}
\item[(i)]
If
 $|2\gamma(1- \beta) + 1 |\leq{1}$,
then
\begin{equation}\label{niru vasu p1 e050}
|a_n|\leq\frac{2|\gamma|(1-\beta)}{(n-1)(1+(n-1)\lambda)}\frac{(\mu+1)(\mu+2)}{(\mu+n)(\mu+n+1)}.
\end{equation}
The inequality (\ref{niru vasu p1 e050})  is  sharp.

\item[(ii)]If
$|2\gamma(1-\beta)+(n-2)|\geq (n-2)$, then
\begin{equation}\label{niru vasu p1 e055}
|a_n|\leq\frac{\prod_{j=0}^{n-2}|2\gamma(1 - \beta)+j|}{(n-1)!(1 + (n-1)\lambda)}\frac{(\mu+1)(\mu+2)}{(\mu+n)(\mu+n+1)}.
\end{equation}
The inequality (\ref{niru vasu p1 e055})  is  sharp.

\item[(iii)]If
 $|2\gamma(1-\beta)+(k-1)|\geq (k-1)$, then
$$
|a_n|\leq\frac{\prod_{j=0}^{k-1}|2\gamma(1 - \beta)+j|}{(n-1)(k-1)!(1 + (n-1)\lambda)}\frac{(\mu+1)(\mu+2)}{(\mu+n)(\mu+n+1)}.
$$
\end{enumerate}
\end{cor}

The following two results give the sharp coefficient bounds for functions in the classes $\mathcal{S}^*(\gamma)$ and $\mathcal{C}(\gamma)$ under
some assumptions.

\begin{cor}\label{niru-vasu cor1}
 Let $f\in \mathcal{S}^*(\gamma)$ be given by (\ref{niru vasu p1 e001}).
\begin{enumerate}
\item[(i)]
If  $|2\gamma + 1 |\leq{1}$, then
\begin{equation}\label{niru vasu p1 e060}
|a_n|\leq \frac{2|\gamma|}{n-1}.
\end{equation}
The equality in (\ref{niru vasu p1 e060}) occurs for the functions $f_n(z)$ where $f_n(z)$ is defined by
\begin{equation*}
f_n(z) = z(1 - z^{n-1})^\frac{-2\gamma}{(n-1)}.
\end{equation*}

\item[(ii)] If
$|2\gamma +(n-2)|\geq (n-2)$, then
\begin{equation}\label{niru vasu p1 e065}
|a_n|\leq \frac{\prod_{j=0}^{n-2}|2\gamma + j|}{(n-1)!}.
\end{equation}
The inequality (\ref{niru vasu p1 e065}) is sharp for the function $f(z)$ where $f(z)$ is defined by
\begin{equation*}
f(z) = \frac{z}{(1-z)^{2\gamma}}.
\end{equation*}
\end{enumerate}
\end{cor}

\begin{cor}\label{niru-vasu cor2}

Let $f\in\mathcal{C}(\gamma)$ be given by (\ref{niru vasu p1 e001}).
\begin{enumerate}
\item[(i)]
If $|2\gamma + 1 |\leq{1}$, then
\begin{equation}\label{niru vasu p1 e070}
|a_n|\leq \frac{2|\gamma|}{n (n-1)}.
\end{equation}
The equality in (\ref{niru vasu p1 e070}) occurs for the functions $f_n(z)$ where $f_n(z)$ is defined by
\begin{equation*}
f'_n(z) = (1 - z^{n-1})^\frac{-2\gamma}{ (n-1)}.
\end{equation*}

\item[(ii)]
If $|2\gamma + (n-2)|\geq (n-2)$, then
\begin{equation}\label{niru vasu p1 e075}
|a_n|\leq \frac{\prod_{j=0}^{n-2}|2\gamma + j|}{n!}.
\end{equation}
The inequality (\ref{niru vasu p1 e075}) is sharp for the function $f(z)$ where $f(z)$ is defined by
\begin{equation*}
f(z)= \int_{0}^{z}\frac{dt}{(1-t)^{2\gamma}}.
\end{equation*}
\end{enumerate}
\end{cor}

It is interesting to note that if we choose $\gamma = 1-\beta$ in Corollaries \ref{niru-vasu cor1} and \ref{niru-vasu cor2} then we can obtain the sharp
coefficient bounds for  functions in the classes $\mathcal{M}(\beta)$ and $\mathcal{N}(\beta)$. In fact these results extend the results  obtained
by Firoz Ali and Vasudevarao \cite{Firoz-Vasudeva-2015}.

\section{Application of Jack Lemma}

In $1999$, Silverman \cite{Silverman-1999} investigated the class $\mathcal{G}_{b}$ for $0<b\leq 1$ which involves the quotient of analytic representations of convexity and starlikeness of a function.
More precisely, for $0 < b \leq 1$, consider the following class

$$
\mathcal{G}_{b} : = \left\{ f \in \mathcal{A} : \left|\frac{1 + zf''(z)/f(z)}{zf'(z)/f(z)} - 1 \right| \leq b\quad\mbox{ for } z \in \mathbb{D}\right\}.
$$
It was proved \cite{Silverman-1999}  that $\mathcal{G}_{b} \subset \mathcal{S}^*(2 /{(1 +\sqrt{1 + 8b})})$.  In 2000, Obradovi\'{c} and Tuneski \cite{Obradovic-Tuneski-2000} improved this result by showing $\mathcal{G}_{b}  = S^* [ 0, -b]\subset S^*(2 /{(1 +\sqrt{1 + 8b})})$. In 2003, Tuneski \cite{Tuneski-2003} found
a nice relation among  $A, B$ and $b$ so  that   functions $f$ in the class  $\mathcal{G}_{b}$  also belong to the class $S^*[A, B]$.
In this paper, we prove a sufficient condition for function $f \in\mathcal{G}_{b}$ to be in the class $ \mathcal{SP}(\alpha)$.

The following lemma, known as Jack lemma, is helpful in proving for our main results.
\begin{lem}\label{niru vasu p1c l001} \cite{Jack-1971}
Let $\omega$ be a non-constant analytic function in the unit disk $\mathbb{D}$ with $\omega(0) = 0$. If  $|\omega(z)|$ attains its maximum value on the circle $|z| = r$ at the point $z_0$ then $z_0 {\omega}'(z_0) = k_0 \omega(z_0)$ and $k_0 \geq 1$.
\end{lem}

The recent applications of Jack lemma we refer to \cite{Mateljevic, Orneck}.
Using the above Jack lemma we prove the following lemma.

\begin{lem}\label{niru vasu p1c l005}
Let  $p$ be an analytic function in the unit disk $\mathbb{D}$ with $p(0) = 1$ and $A = e^{-2i\alpha}$ be a complex constant with $|\alpha|< \pi/2.$  If $p$ satisfies the following condition

\begin{equation}\label{niru vasu p1c e005}
\frac{z p'(z)}{p^2(z)}\prec \frac{(A + 1)z}{(1 + Az)^2} := h_1(z) \quad \mbox{for} \quad z\in\mathbb{D}
\end{equation}
then
\begin{equation}\label{niru vasu p1c e010}
p(z) \prec \frac{1 + Az}{1 - z} \quad \mbox{for} \quad z\in\mathbb{D},
\end{equation}
that is, $p \in \mathcal{SP}(\alpha).$
\end{lem}

\begin{proof}

Let $p(z) = {(1 + A \omega(z))}/{ (1 - \omega(z))}.$ Then $\omega$ is analytic in $\mathbb{D}$ and  $\omega(0) = 0$. A simple computation shows that
$$
\frac{z p'(z)}{p^2(z)} = \frac{(A + 1)z \omega'(z)}{(1 + {\omega}(z))^2} \quad \mbox{for} \quad z\in\mathbb{D}.
$$
Now the subordination relation ({\ref{niru vasu p1c e010}}) holds if and only if $|\omega(z)| < 1$ for  $z$ in $\mathbb{D}$. Assume that there exists a point $z_0\in \mathbb{D}$ such that  $|\omega(z_0)| = 1$. Then by Jack lemma,  $z_0 \omega'(z_0) = k_0 \omega(z_0)$ and $k_0\geq 1$. For such $z_0$ we have $z_0 p'(z_0)/ p^2(z_0) = k_0 h_1(\omega(z_0))$ which does not contain in $h_1(\mathbb{D})$ because $|\omega(z_0)| = 1$ and $k_0 \geq 1.$ This contradicts the subordination condition ({\ref{niru vasu p1c e005}}).
Hence $|\omega(z)| < 1$ for all $z\in \mathbb{D}$ which yields the desired result.
\end{proof}

Using Lemma \ref{niru vasu p1c l005} we prove the following theorem.

\begin{thm}\label{niru vasu p1c t001}
Let $f\in \mathcal{A}$ and  $A = e^{-2i\alpha}$ be a complex constant with $|\alpha|< \pi/2.$ If
$$
\frac{1 + zf''(z)/f'(z)}{zf'(z)/f(z)}\prec 1 + \frac{(1 + A)z}{(1 + Az)^2} \quad \mbox{for} \quad z\in\mathbb{D}
$$
then $f\in \mathcal{SP}(\alpha)$.
\end{thm}

\begin{proof}

Let $p(z) = \frac{zf'(z)}{f(z)}$. Then $p$ is analytic in $\mathbb{D}$ and $p(0)= 1$. A simple computation shows that
$$
\frac{zp'(z)}{p^2(z)} = \frac{1 + zf''(z)/f'(z)}{zf'(z)/f(z)} - 1\prec\frac{(1 + A)z}{(1 + Az)^2} \quad \mbox{for} \quad z\in\mathbb{D}.
$$
In view of Lemma \ref{niru vasu p1c l005}, it follows that $p(z)\prec (1 + Az)/ (1 - z)$ and hence $f\in \mathcal{SP}(\alpha)$.

Using Theorem \ref{niru vasu p1c t001}, we obtain the following result.
\end{proof}

\begin{cor}
Let  $A =e^{-2i\alpha}$ be a complex constant with $|\alpha|< \pi/2.$ Then $\mathcal{G}_b \subset \mathcal{S}^*[A, -1] := \mathcal{SP}(\alpha)$ when
$b = |1 + A|/ 4 .$
\end{cor}

\begin{proof}

For $f\in \mathcal{G}_b$, we have
$$
\frac{1 + zf''(z)/f'(z)}{zf'(z)/f(z)}\prec  1 + bz \quad \mbox{for} \quad z\in\mathbb{D}.
$$
Let $h_2(z) = 1+\frac{(1 + A)z}{(1 + Az)^2}.$
Then a simple computation shows that
$$
\min\{|h_2(e^{i\theta}) - 1| : \theta \in [0, 2\pi)\} = \frac{|1 + A|}{4}.
$$
If $b = |1 + A|/4 $ then by using the definition of subordination we obtain $1 + bz \prec h_2(z)$. Therefore  from Theorem \ref{niru vasu p1c t001}, it follows that $f \in \mathcal{S}^*[A, -1] := \mathcal{SP}(\alpha).$

\end{proof}

\subsection{Starlike  univalent functions   of order   $\alpha$}
Let $B(z_0;r)$ denote the open ball centred at $z_0$ and radius $r$.  We say that   $f\in \mathcal{H}(\alpha)$, $0< \alpha <1$,   if  $f\in \mathcal{A}$ and
$A_\alpha(z)= \frac{2 \alpha f(z)}{z f'(z)}$  maps the unit disk  $\mathbb{D}$ into
$B(1;1)$.
Since the conformal mapping $B(w)= (1+w)^{-1}$  maps   $\mathbb{D}$  onto
 ${\rm Re\,} w >1/2$,  one can see that  the classes $\mathcal{S}^*(\alpha)$  and
$\mathcal{H}(\alpha)$ coincide.

Let   $f\in \mathcal{H}(\alpha)$    and consider the function
$h(z):=h_\beta(z)=(\frac{z}{f(z)})^\beta-1$, where  $0< \beta \leq 1$.
If $f\in \mathcal{H}(\alpha)$, $1/2 \leq \alpha <1$, using   Jack's lemma, \"{O}rnek \cite{Orneck}
showed that $h$   satisfies  the condition of  the Schwarz lemma:
$h$  maps  $\mathbb{D}$ onto itself  and $h(0)=0$,   and  he has   proved

\begin{lem}
Let  $f\in\mathcal{H}(\alpha)$, $1/2 \leq \alpha <1$ and $1/\beta=2(1-\alpha)$. Then
\begin{enumerate}
\item[(i)] $\displaystyle |f(z)|\leq  \frac{|z|}{(1-|z|)^{1/\beta}}$\\[0.mm]
\item[(ii)]  $|f''(0)|\leq 2/\beta$.
\end{enumerate}
For $\beta=1$, we find
\begin{enumerate}
\item[(i')] $\displaystyle |f(z)|\leq  \frac{|z|}{(1-|z|)}$\\[0.mm]
\item[(ii')] $|f''(0)|\leq 2$.
\end{enumerate}
\end{lem}

\begin{example}
Let   $k_\beta(z)=z (1 +z)^{-1/\beta}$, $0< \beta \leq 1$.
Then $\frac{zk'_\beta(z)}{k_\beta(z)}=A_\beta$, where  $A_\beta(z)= 1 - \frac{1}{\beta}\frac{z}{1+z}$.
Since $A_\beta$   maps   $\mathbb{D}$  onto  ${\rm Re\,} w >1 - \frac{1}{2\beta}$.  One can see that   $ k_\beta$ belongs
$\mathcal{S}^*(\alpha)$ if and only if  $\beta \geq \frac{1}{2 (1- \alpha)} $. If  $1/\beta >2$ then $k_\beta$  is not univalent in $\mathbb{D}$.
\end{example}

The subject related  to Jack's lemma  has been discussed  by \"{O}rnek \cite{Orneck} in
a recent paper. Recently, Mateljevi\'c \cite{Mateljevic} has  extended \"{O}rnek's  result and obtained  the following.

\begin{thm}
If $f$ belongs  $\mathcal{S}^*(\alpha)$, $0\le \alpha <1$,     and
$1/\beta=2(1-\alpha)$, then

\begin{enumerate}
\item[(i)] $\displaystyle|f(z)|\leq  \frac{|z|}{(1-|z|)^{1/\beta}}$
\item[(ii)]  $|f''(0)|\leq 2/\beta$.
\end{enumerate}
\end{thm}

In particular,  it can be seen that   Ornek's result   (i')
$|f(z)|\leq  \frac{|z|}{(1-|z|)}$  and
(ii')$ |f''(0)|\leq 2$ if $f$ belongs  to  the class $\mathcal{S}^*(1/2)$. For  convex functions (i')  holds.
Since convex functions are in  $\mathcal{S}^*(1/2)$,  this  result  is a
generalization  of corresponding one for convex functions.

\vspace{1cm}
\noindent\textbf{Acknowledgement:}  The authors thank the referee for useful comments and suggestions. The first author thanks UGC for financial support.

\end{document}